\documentclass[11pt]{amsart}
\usepackage{graphicx}
\usepackage{amsmath}
\usepackage{amssymb}
\usepackage[all]{xy}
\usepackage{hyperref}

\numberwithin{equation}{subsection}
\newtheorem{theorem}[subsection]{Theorem}
\newtheorem{classification-theorem}[subsection]{Classification Theorem}
\newtheorem{decomposition-theorem}[subsection]{Decomposition Theorem}
\newtheorem{definition}[subsection]{Definition}
\newtheorem{periodicity-conjecture}[subsection]{Periodicity Conjecture}
\newtheorem{lemma}[subsection]{Lemma}
\newtheorem{proposition}[subsection]{Proposition}
\newtheorem{corollary}[subsection]{Corollary}

\newtheorem{assumption}[subsection]{Assumption}
\newtheorem{example}[subsection]{Example}
\newtheorem{remark}[subsection]{Remark}

\newcommand{\res}{\operatorname{res}\nolimits}

\newcommand{\Gr}{\operatorname{Gr}\nolimits}

\newcommand{\Hom}{\operatorname{Hom}\nolimits}
\newcommand{\im}{\operatorname{im}\nolimits}
\newcommand{\cok}{\operatorname{cok}\nolimits}

\newcommand{\coker}{\operatorname{Coker}\nolimits}
\newcommand{\Ext}{\operatorname{Ext}\nolimits}

\newcommand{\Mod}{\operatorname{Mod }\nolimits}
\renewcommand{\mod}{\operatorname{mod}\nolimits}

\newcommand{\rep}{\operatorname{rep}\nolimits}
\newcommand{\proj}{\operatorname{proj}\nolimits}
\newcommand{\inj}{\operatorname{inj}\nolimits}

\newcommand{\ind}{\operatorname{ind}\nolimits}

\newcommand{\Z}{\operatorname{\mathbb{Z}}\nolimits}
\newcommand{\C}{\operatorname{\mathbb{C}}\nolimits}
\newcommand{\N}{\operatorname{\mathbb{N}}\nolimits}

\newcommand{\gdim}{\operatorname{\underline{dim}}\nolimits}

\newcommand{\GL}{\operatorname{GL}\nolimits}

\newcommand{\iso}{\stackrel{_\sim}{\rightarrow}}

\newcommand{\id}{\mathbf{1}}

\newcommand{\cc}{{\mathcal C}}
\newcommand{\cd}{{\mathcal D}}

\newcommand{\cl}{{\mathcal L}}
\newcommand{\cm}{{\mathcal M}}

\newcommand{\cp}{{\mathcal P}}
\newcommand{\cR}{{\mathcal R}}

\newcommand{\cS}{{\mathcal S}}
\newcommand{\ct}{{\mathcal T}}

\newcommand{\eps}{\varepsilon}

\renewcommand{\tilde}[1]{\widetilde{#1}}

\begin{document}
\title[Generalized quiver varieties and triangulated categories]{Generalized quiver varieties and triangulated categories}
\author{Sarah Scherotzke}
\address{S.~S.~: University of M\"unster, Mathematisches Institut, 
Einsteinstrasse 62, 48149 M\"unster, Germany}

\email{scherotz@math.uni-muenster.de}


\begin{abstract}
In this paper, we introduce \emph{generalized quiver varieties} which
include as special cases classical and cyclic quiver varieties. 
The geometry of generalized quiver varieties is governed by a finitely generated algebra $\cp$: the algebra $\cp$ is self-injective if the quiver $Q$ is of Dynkin
type, and coincides with the preprojective algebra in the case of classical
quiver varieties. We show that  in the Dynkin case the strata of generalized quiver varieties
are in bijection with the isomorphism classes of objects in $\proj \cp$, and
that their degeneration order coincides with the
Jensen--Su--Zimmermann's degeneration order on the triangulated cate-
gory $\proj \cp$. Furthermore, we prove that classical quiver varieties of type $A_n$ can be realized as moduli spaces of representations of an algebra $\cS$.

\end{abstract}

\maketitle

\tableofcontents

\section{Introduction}
Quiver varieties associated with
a finite and acyclic quiver $Q$ were first introduced by Nakajima in \cite{Nakajima94}.
Since then they have been of great importance in Nakajima's geometric study
of Kac-Moody algebras and their representations \cite{Nakajima94}, \cite{Nakajima98}. 
Furthermore, Nakajima quiver varieties yield important examples of symplectic hyper-K{\"a}hler 
manifolds appearing in various fields of representation theory. 

Graded quiver varieties, defined as fixed point sets of ordinary quiver varieties in \cite{Nakajima01}, found an interesting 
application in constructing a monoidal categorification of cluster algebras (see \cite{Nakajima11}, \cite{Qin12}, \cite{KimuraQin12}) in the sense of 
Hernandez and Leclerc, \cite{HernandezLeclerc10} and \cite{HernandezLeclerc11}. Cyclic quiver varieties have recently been used by 
Qin to realize a geometric construction of quantum groups \cite{Qin14}. 
In \cite{KellerScherotzke13a}, motivated by the construction of monoidal categorifications of cluster algebras, 
we studied various properties of graded quiver varieties. 

 In this paper we introduce \emph{generalized quiver varieties} and investigate their properties. 
 Classical and cyclic quiver varieties arise as special cases of  generalized quiver varieties, and our theory provides a new construction of these varieties in terms of representations of orbit categories. Many properties of classical and cyclic quiver varieties (such as smoothness) are general features  of  all generalized quiver varieties.  This allows us to reprove some known results in a way that applies simultaneously to classical and cyclic quiver varieties, removing the need to study these two cases  separately.  We also prove results, which are new for classical quiver varieties of Dynkin type, such as the existence of a stratification functor. 
 
%
%

One of the main motivations for this project is that the theory of generalized quiver varieties provides the tools for comparing Bridgeland's and Qin's work on quantum groups and Hall algebras \cite{Bridgeland13}, \cite{Qin14}. This is accomplished in our recent work \cite{ScherotzkeSibilla} that  relies in an essential way on the results of this paper. 
Another source of motivations comes from the problem of desingularizing quiver Grassmannians. We show  in \cite{S} that
generalized quiver varieties give the right framework to construct desingularization
maps for quiver Grassmannians of self-injective algebras of finite type. 
 
 In the rest of this Introduction we give a more detailed account of our main results. 
 
 \subsection{Nakajima quiver varieties}
 
Recall that to the choice of dimension vectors $w$ and $v$ for the set of vertices of $Q$, 
Nakajima associates the smooth quiver variety $\cm(v,w)$ and the affine quiver variety $\cm_0(v,w)$. 
There are two main features of quiver varieties which play an important role in their application to Kac-Moody algebras and cluster algebras:   If $Q$ is of Dynkin type, i.e. a quiver whose underlying graph is an ADE Dynkin diagram,
\begin{enumerate}
\item there is a projective map $\pi: \cm(v,w) \to \cm_0(v, w)$. 
\item  the map $\pi$ induces a stratification $$\cm_0(w)= \bigsqcup_v \cm_0^{reg}(v,w)$$ into non-empty smooth and locally closed strata. 
\end{enumerate} 

In this paper we show that our generalized quiver varieties  of Dynkin type satisfy analogous properties. Furthermore, we show that their geometry is governed by 
an associated self-injective algebra. This greatly generalizes work of \cite{KellerScherotzke13a}. In the case of classical Nakajima quiver varieties, the self-injective algebra describing the geometry of quiver varieties is
the preprojective algebra $\cp_Q$, associated to the quiver $Q$. Preprojective algebras were introduced in \cite{GelfandPonomarev79} and play an important role in Lusztig's study of canonical bases \cite{Lusztig91}, \cite{Lusztig92}.  
 
\subsection{Generalized Nakajima categories and generalized quiver varieties} 

In order to define    generalized  quiver varieties 
we consider configurations of the regular (graded) Nakajima 
categories $\cR^{gr}_C$, which are certain quotient categories of 
the Nakajima category $\cR^{gr}$, and were introduced in \cite{KellerScherotzke13a}.  
Taking the orbit category with respect to an isomorphism $F$ of $\cR^{gr}_C$ yields the \emph{generalized regular Nakajima category} $\cR$.  
In analogy to the classical case, we can define generalized quiver varieties associated to $\cR$ using geometric invariant theory and obtain 
two varieties $ \cm(v,w)$ and $ \cm_0(v,w)$  satisfying the properties $(1)$ and $(2)$ of Section 1.1. The cyclic and classical quiver varieties appear as special case of this construction. 

Associated to the regular Nakajima category $\cR$, we consider a certain full subcategory $\cS$ of $\cR$, the \emph{singular Nakajima category}.
We show that the quotient category $\cp \cong \cR/ \langle \cS \rangle$
 describes the stratification into smooth strata $\cm^{reg}_0(v,w)$, the degeneration order between the strata and the fibres of $\pi$.  
In fact, if we consider the special case of classical quiver varieties, $\cp$ is isomorphic to the preprojective algebra $\cp_Q$. 
If $Q$ is a Dynkin quiver the category $\cp$ has similar properties than the preprojective algebra: 
it is self-injective, finite-dimensional and its category of 
projective modules is triangulated. The triangulated structure of projective $\cp$--modules 
will be crucial to establish an equivalence between the degeneration 
order on strata and a degeneration order on triangulated categories defined by Jensen, Su and Zimmermann  \cite{JensenSuZimmermann05a}.

Finally, note that graded Nakajima categories and graded quiver varieties are not a special case 
of generalized Nakajima category respectively generalized quiver varieties.

\subsection{Affine quiver varieties as moduli spaces} 

It was shown in \cite{KellerScherotzke13a} and \cite{LeclercPlamondon12} that 
graded affine quiver varieties can be realized as spaces of representations
of the (graded) singular Nakajima category $\cS^{gr}$.  Let $Q$ be of Dynkin type. 
As shown in Theorem \ref{thm:An}, an analogous result holds for the classical affine quiver varieties
associated with quivers $A_n$ that is, the affine quiver variety $\cm_0(w)$ is
isomorphic to $\rep(w,\cS)$, the space of representations with dimension vector
$w$ of the singular Nakajima category $\cS$. In general, we only obtain a bijection
of closed points $\res :\cm_0(w) \to \rep(w, \cS)$. We also provide an example
in which the equivalence of schemes fails. Under strong assumptions on $F$,
we do obtain the equivalence of schemes $\rep(w, \cS)$ and $\cm_0(w)$.
We give a more explicit description of $\cS$. We have that  $\cS$
is given as an algebra by $\C Q_\cS/ I$, where $Q_\cS$ is a finite quiver and $I$ is an
admissible ideal of the path algebra of $Q_{\cS}$. We determine the quiver $Q_\cS$
and the number of minimal relations of paths between two vertices of $Q_\cS$
in terms of dimensions of morphism spaces in $\proj \cp$  (see Proposition \ref{proposition: quiver of S}).
Hence the self-injective algebra $\cp$ plays a key role in the description of the affine quiver variety.

\subsection{Stratification of affine quiver varieties} 
One of our main results  consists in the construction of a so-called \emph{stratification map} $\Psi$ from the category
of finite-dimensional $\cS$--modules to the category $\inj^{nil} \cp$ of finitely-cogenerated
injective $\cp$--modules that have nilpotent socle. We call $\Psi$ a stratification map as it satisfies the following properties:
\begin{itemize}
\item  two points $M$ and $M'$ of $\cm_0(v,w)$ belong to the same stratum if and only if their images under $\Psi \circ \res$ are isomorphic (Theorem~\ref{thm: stratification-dynkin} and \ref{thm: stratification}). 
\item If $Q$ is of Dynkin type, $\inj^{nil} \cp$ equals $ \proj \cp$, the category of finitely-generated projective $\cp$--modules which has a triangulated structure.
In this case, we show that $\Psi$ can be promoted to a $\delta$--functor and behaves well with respect to the degeneration order on strata: the degeneration order in $\cm_0(w)$ 
corresponds under the stratification functor to the degeneration order in the sense of \cite{JensenSuZimmermann05a} applied to the triangulated category $\proj \cp$. 

\end{itemize}

{\bf Acknowledgements: } The author would like to thank Bernhard Keller for his many valuable suggestions, 
Nicol{\`o} Sibilla, Bernard Leclerc  and Pierre-Guy Plamondon for comments on the first draft, Peter Tingley,
 Alastair Savage and Alfredo N{\'a}jera Ch{\'a}vez for useful discussions. 
The author is indebted to Hiraku Nakajima for inviting her to the RIMS. Finally, the
author thanks the referee for his useful corrections and remarks.

\section{Generalized Nakajima categories}\label{s:Nakajima-categories}
\subsection{Notations}
For later use, we introduce the following notations. 
Let $k$ be an algebraically closed field of characteristic zero and $\Mod k$ be the category of $k$-vector spaces. 
Recall that a {\em $k$-category} is a category whose
morphism spaces are endowed with a $k$-vector space structure such
that the composition is bilinear. Let $\cc$ be a $k$-category and let
$\Mod(\cc)$ be the category of {\em right $\cc$--modules}, i.e.~$k$-linear functors 
$\cc^{op} \to \Mod(k)$. For each object $x$ of $\cc$, we obtain a {\em free
module}
\[
x^{\wedge}=x^{\wedge}_\cc = \cc(?,x): \cc^{op} \to \Mod k
\]
and a {\em cofree module}
\[
x^{\vee}= x^\vee_\cc = D (\cc(x,?)): \cc^{op} \to \Mod k.
\]
Here, we write $\cc(u,v)$ for the space of morphisms $\Hom_\cc(u,v)$ and
$D$ for the duality over the ground field $k$. Recall that for each object $x$ of $\cc$ and
each $\cc$--module $M$, we have canonical isomorphisms
\begin{equation} \label{eq:Yoneda}
\Hom(x^\wedge,M) = M(x) \quad\mbox{and}\quad
\Hom(M,x^\vee) = D(M(x)).
\end{equation}
In particular, the module $x^\wedge$ is projective and $x^\vee$ is injective. We will denote $\proj \cc$ the full subcategory of  $\cc$--modules with
objects the finite direct sums of objects $x^{\wedge}$ and dually, we denote $\inj \cc$ the full subcategory of  $\cc$--modules with
objects the finite direct sums of objects $x^{\vee}$. 

Furthermore, we denote throughout the paper by $\cc_0$ the set of objects of $\cc$ and mean by a dimension vector of $\cc$ a function $w: \cc_0 \to \N$ with
finite support. We define $\rep(\cc, w)$ to be the space of $\cc$--modules $M$ such that $M(u) = k^{w(u)}$ for each
object $u$ in $\cc_0$. Note that all $k$-linear categories that we consider will be basic. 

Finally, we call a $\cc$--module $M$ {\em pointwise
finite-dimensional} if $M(x)$ is finite-dimensional for each object $x$ of $\cc$. 
For all triangulated categories $\ct$, we shall denote by $\Sigma$ the shift functor. We will denote by $\tau$ the Auslander-Reiten translation and by $S$ the Serre functor of $\ct$, if they exist. 
We will denote by $\cd_Q$ the bounded derived category of finite-dimensional $kQ$--modules. Note also that we  always compose arrows from left to right. 

\subsection{Nakajima Categories and Happel's Theorem}\label{Happel}
In this section, we define the generalized regular (resp. singular) Nakajima categories $\cR$ (resp. $\cS$). The regular Nakajima category is a mesh category which we can associate to any acyclic finite quiver and $\cS$ is a full subcategory of $\cR$. The graded Nakajima categories have also been defined in \cite{KellerScherotzke13a} and the generalized Nakajima categories in \cite{S}, but for the convenience of the reader, we give a short definition here. In \cite{KellerScherotzke13a}, we denoted the graded Nakajima categories by $\cR$ and $\cS$, here we will denote them by $\cR^{gr}$ and $\cS^{gr}$ to distinguish them from the ungraded Nakajima categories.

Let $Q$ be a finite acyclic quiver with set of vertices $Q_0$ and set of arrows $Q_1$. We will assume throughout the article that 
$Q$ is connected. This allows a simplified exposition of results as we will have to distinguish between the case when $Q$ is of Dynkin type,
 that is the unoriented underlying graph of $Q$ is a Dynkin diagram, or not.

The {\em framed quiver $Q^f$} is obtained from $Q$ by adding, for each vertex $i$, a new vertex $i'$ and
a new arrow $i \to i'$. For example, if $Q$ is the quiver $1 \to 2$, the
framed quiver is
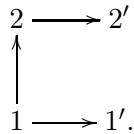
\[ 
\xymatrix{ 
2 \ar[r] & 2' \\
1 \ar[r] \ar[u] & 1'.
} 
\] 
Let $\Z Q^f$ be the repetition quiver of $Q^f$ (see  \cite{Riedtmann80a}). We refer to the
vertices $(i', p)$, $i\in Q_0$, $p\in \Z$, as the {\em frozen vertices} of $\Z Q^f$. We define the automorphism $\tau$ of $\Z Q$ to be the shift
by one unit to the left, so that we have in particular $\tau(i,p)=(i, p-1)$ for all vertices $(i, p) \in Q_0 \times \Z$. Similarly, 
we define $\sigma$ by $(i,p) \mapsto (i', p-1)$ and $ (i', p) \mapsto (i, p)$  for all $i \in Q_0$, $p \in \Z$. 

For example, if  $Q$ is the Dynkin diagram $A_2$, the repetition $\Z Q^f$ is the quiver
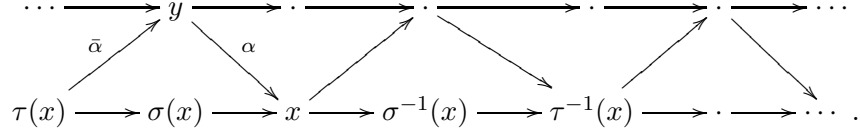
\[ 
\xymatrix{ \cdots  \ar[r]  & y \ar[r] \ar[rd]^{ \alpha}& \cdot \ar[r]  &  \cdot \ar[r]   \ar[rd] &  \cdot \ar[r] & \cdot \ar[r]   \ar[rd] & \cdots \\
\tau(x)  \ar[r]  \ar[ru]^{\bar \alpha}  & \sigma(x)  \ar[r]  &  x  \ar[r]  \ar[ru]  &  \sigma^{-1} (x) \ar[r]  & \tau^{-1}(x) \ar[r]  \ar[ru]  &  \cdot \ar[r]  &  \cdots\ . }
\] 

Finally we denote by $\bar \alpha$ the arrow associated by construction to $\alpha: y \to x$ that runs from $\tau x \to y$. 
As in  \cite{Gabriel80} and \cite{Riedtmann80a}, we denote by $k(\Z Q)$ the \emph{mesh category} of $\Z Q$,
 that is the objects are given by the vertices of $\Z Q$ and the morphisms are 
$k$-linear combinations of paths modulo the ideal spanned by the mesh relations 

$$R_x: \sum_{\alpha: y \to x}  \overline {\alpha} \alpha,$$ where the sum runs through all arrows of $\Z Q$ ending  in $x$. Note that we always compose arrows from left to right.

By Happel's Proposition~4.6 of \cite{Happel87}
and Theorem~5.6 of \cite{Happel88}, there is a fully faithful embedding 

$$ H : k(\Z Q) \hookrightarrow \ind \cd_Q$$ where $\ind \cd_Q$ denotes the category of indecomposable complexes in 
the bounded derived category of mod $kQ$. The functor $H$ is an equivalence if and only if $Q$ is of Dynkin type. 

The action of $\tau$ on $\Z Q$ corresponds to the action of the Auslander-Reiten translation
on $\cd_Q$, hence justifies our notation. 

\begin{definition}
The {\em regular  graded Nakajima category $\cR^{gr}$} has as objects the vertices of $\Z Q^f$ and the morphism space from
$a$ to $b$ is the space of all $k$-linear combinations of paths from
$a$ to $b$ modulo the subspace spanned by all elements $u R_x v$, where
$u$ and $v$ are paths and $x$ is a non-frozen vertex. 
The {\em singular graded Nakajima category
$\cS^{gr}$} is the full subcategory of $\cR^{gr}$ whose objects are the frozen 
vertices. 
\end{definition}

Note that $\cR^{gr}$ is not equivalent to the mesh category $k (\Z Q^f )$, as we do not impose mesh relations in frozen objects.

\subsection{Configurations and orbit categories of Nakajima categories}\label{ss: configurations}
Let $C$ be a subset of the set 
of vertices of the repetition quiver $\Z Q$. We denote $\cR^{gr}_C$ the
quotient of $\cR^{gr}$ by the ideal generated by the identities of the
frozen vertices not belonging to $\sigma(C)$ and by $\cS^{gr}_C$ the
full subcategory of $\cR^{gr}_C$ formed by the objects corresponding to the 
vertices in $\sigma(C)$.
Furthermore, we denote by $\Z Q_C$ the quiver obtained from $\Z Q$ 
by adding for all $x \in C$ a vertex $\sigma(x)$ and arrows
$\tau (x) \to \sigma(x)$ and $\sigma(x) \to x$. 

Let $F$ be a triangulated isomorphism on $\cd_Q$. As $k(\Z Q)$ is the mesh category
of the preprojective objects in $\cd_Q$, $F$ induces a $k$-linear isomorphism
of $k(\Z Q)$, which we also denote by $F$. We make the following assumption
on $C$ and $F$.

\begin{assumption} \label{main-assumption}
For each vertex $x$ of $\Z Q$, the sequences
\begin{equation} \label{eq:left-exact-sequences}
0 \to \cR^{gr}_C(?,x) \to \bigoplus_{x \to y} \cR^{gr}_C(?,y) \quad \mbox{and}\quad
0 \to \cR^{gr}_C(x,?) \to \bigoplus_{y \to x} \cR^{gr}_C(y,?) \quad
\end{equation}
are exact, where the sums range over all arrows of $\Z Q_C$ whose
source (respectively, target) is $x$.

Furthermore,  we have that $ F(C) \subset C$ and $F^n $ is not the identity on objects for all $n \in \Z$. 
\end{assumption}
We call all $C$ satisfying the above assumption \emph{admissible configuration} and $(C, F)$ an \emph{admissible pair.} 
For example if $C$ is the set of all vertices of $\Z Q$, then $C$ is  an admissible. Another class of examples is given in \cite{LeclercPlamondon12}.

Note that $F$ commutes with $\tau$ and extends uniquely to $\cR^{gr}_C$ by setting $F\sigma(c):= \sigma(F(c))$ for all $c \in C$. 
We will denote by $\cR$ the orbit category of $\cR^{gr}_C/ F$ and by $\cS$ its full subcategory 
with objects $\sigma(C)$. 
Note that the quotient $ \cR^{gr}_C / \langle  \cS^{gr}_C \rangle \cong k(\Z Q)$ does not depend on 
$C$ and is equivalent to the image of the Happel functor $H$. 
It is easy to see that the orbit category of $ \cR^{gr}_C / \langle  \cS^{gr}_C \rangle$ by $F$ is naturally equivalent to the quotient 
$  \cp:= \cR / \langle  \cS \rangle $.

Finally, we denote by $\tilde Q$ the quiver $\Z Q_C/F$ with vertices the $F$--orbits on the set of vertices of $\Z Q_C$ 
and the number of arrows $x \to y$ between two fixed representatives $x$ and $y$ of $F$--orbits is
given by the number of arrows from $x \to F^i y$ for all $i \in \Z$ in the quiver $\Z Q_C$. By our condition on $F$, 
it is easy to see that $\tilde Q$ is a finite quiver and the canonical map $ \Z Q_C \to \tilde Q$ is a Galois covering. 

In the sequel, we will identify the orbit categories $\cR$, $\cS$ and $\cp$ with their equivalent skeletal categories, in which we identify 
all objects lying in the same $F$--orbit. 

\subsection{The Dynkin case} 

In this section, we assume that $Q$ is of Dynkin type. We will see that choosing an automorphism $F$ as above is equivalent to the choice
of a triangulated automorphism of $\cd_Q$ such that the canonical map $ \cd_Q \to \proj \cp$ is a triangulated functor. 

\begin{example}\label{exa: A2}
In the case that $Q=A_2$, $F=\tau$  and that $C$ is the set of all vertices of $\Z Q$, our assumptions are satisfied and  $ \cR$ is equivalent 
to the path category of $\tilde Q$ 
\[ 
\xymatrix{ 
2 \ar[r]^{\alpha} \ar@/^/[d]^{\overline \gamma}& 2' \ar@/^/[l]^{\overline \alpha}\\
1 \ar[r] ^{\beta}\ar[u] ^{ \gamma}& 1'  \ar@/^/[l] ^{\overline \beta}
} 
\]
modulo the mesh relations, which are given by 
\[
\alpha \overline \alpha+ \overline \gamma  \gamma=0 \text{ and } \beta \overline \beta+  \gamma \overline \gamma=0.
\]
We see that $ \cp$ is equivalent to the projective modules over the preprojective algebra $\cp_{A_2}$.
\end{example}
We have just seen an example in which $\cp$ is equivalent to the projective modules over the preprojective algebra, which we will recall in \ref{ss: preprojective}. Note that $\proj \cp_Q$ is triangulated. 
This holds in greater generality as shown in \cite{S}. 

\begin{lemma}\label{lemma: tilde P}
Suppose that $Q$ is of Dynkin type. 
Under the assumption  \ref{main-assumption}, $F$ lifts to a triangulated functor of $ \cd_Q$ such that the canonical morphism $\cd_Q \to \proj \cp $ is triangulated. Furthermore
 $ \cp$ is $\Hom$-finite and has up to isomorphism only finitely many objects and $\proj  \cp= \inj  \cp$.
\end{lemma}

Note that the category $\proj \cp$ also admits Auslander-Reiten triangles. The existence of Auslander-Reiten triangles is equivalently to the existence of a Serre functor by Proposition I.2.3 of  \cite{BerghReiten}. The Serre functor of $\cd_Q$ induces a Serre functor of the orbit category as it is $\Hom$-finite. 
Its Auslander-Reiten quiver is given by $\Z Q/F$ and $\proj \cp$ is standard, that is 
the mesh category $k(\Z Q)/F\cong k (\Z Q/F) $ is equivalent to $\cp$.

\begin{example}\label{cluster}
We consider the example, where $Q= A_2$ and $C = (\Z Q)_0$. Let us consider the triangulated functor $F= \Sigma \tau$ acting on 
$\cd_Q$ and by Happel's Theorem on the mesh category $k(\Z Q)$. Clearly $F$ satisfies all conditions in \ref{main-assumption} and $ \cR:=\cR^{gr}/ F$ is given by the path category 
to the quiver $\tilde Q$ obtained by identifying the two vertices labelled by $S_1$ and $P_2$ respectively
\[ \xymatrix
{ & P_2 \ar[r] \ar[dr]_{\alpha} &  \sigma(\Sigma S_2) \ar[r] & \Sigma S_1  \ar[r] \ar[dr] & \sigma( S_1) \ar[r] & S_1 \ar[rd] \\
S_1 \ar[ur]_{\bar \alpha} \ar[r]_{\bar f} & \sigma(S_2) \ar[r]_f  &S_2 \ar[ur] \ar[r] & \sigma(\Sigma P_2) \ar[r] & \Sigma P_2  \ar[r] \ar[ur] & \sigma (P_2) \ar[r] & P_2 
}
\]  The relations in $ \cR$ are given by the mesh relations in the non-frozen objects corresponding to the vertices $S_1, P_2, S_2, \Sigma S_1, \Sigma P_2$. 
For example the mesh relation in $S_1$ yields $$\bar \alpha \alpha+\bar f f=0.$$
Furthermore $\cd_Q/ F $ is the cluster category associated to $Q$ introduced in \cite{BuanMarshReitenReineckeTodorov} which is triangulated by \cite{Keller05}. 
\end{example}

\subsection{Kan extensions and Stability} \label{ss:Kan extensions and Stability}
In \cite{KellerScherotzke13b}, we introduced in more generality the notion of intermediate extensions
and stability. We recall them here in a version which is adapted to the setup of this paper. 
We call an $\cR$--module $M$ {\em stable} if $\Hom_{\cR}(S, M)$ vanishes 
for all modules $S$ supported only in non-frozen vertices. Equivalently, $M$ does not contain
any non zero submodule supported only on non-frozen vertices. 
We call $M$ {\em costable} if we
have $\Hom_{ \cR}(M, S)=0$ for each module $S$ supported only on non-frozen vertices.
Equivalently, $M$ does not admit
any non zero quotient supported only on non-frozen vertices. 
A module is {\em bistable}, if it is both stable and costable. 

As the restriction functor $$\res: \Mod \cR \to \Mod \cS$$ is 
a localization functor in the sense of \cite{Gabriel62}, it admits a right and a left adjoint which we denote $K_R$ and $K_L$ respectively: the left and right Kan extension cf.~\cite[Chapter X]{MacLane98}.

We obtain the following recollement of abelian categories: 
\[ 
\xymatrix{
 \Mod \cp \ar[rr]|\;  && \; \Mod \cR \; \ar@<2ex>[ll] \ar@<-2ex>[ll] \ar[rr]|>>>>>>>>>{\res}  && \; 
 \Mod  \cS
\ar@<2ex>[ll]^{K_R } \ar@<-2ex>[ll]_{K_L }.
}
\]

We define the intermediate extension $$K_{LR}: \Mod \cS \to \Mod   \cR$$ as the image of the canonical map 
$K_R \to K_L$ (see \cite{KellerScherotzke13b} for general properties). To distinguish the graded from the non graded case, we will denote 
the Kan extensions by $K_R^{gr} $, $K_L^{gr} $ and $K_{LR}^{gr} $ in the graded setting. 

\begin{remark} 
Let $e$ be the idempotent in $\cR$ associated to the objects of $\cS$. Then $ e \cR e\simeq \cS$ and $K_L = \cR e  \otimes_\cS -$ and $K_R = \Hom_{\cS}(e \cR,-)$. If $Q$ is of Dynkin type, then $\cp$ is finite-dimensional and therefore $\cR e$ and $e \cR$ are finitely generated $\cS$-modules. Hence $K_R$ and $K_L$ map finite-dimensional $\cS$-modules to finite-dimensional $\cR$-modules.
\end{remark}

In the next Lemma, we summarize some basic properties of Kan extensions used in the sequel of this article. They are proven in Lemma 2.2 
of \cite{KellerScherotzke13b}.
\begin{lemma}\label{lemma:stable}
Let $M$ be a $\cR$--module, then the following holds. 
\begin{enumerate}

\item The adjunction morphism $M\to K_R \res M$ is injective if and only if $M$ is stable.

\item Dually, $M$ is costable if  and only if the adjunction morphism $K_L \res M\to  M$ is surjective.

\item $M$ is bistable if and only if $K_{LR} \res M \cong M$.

\item The adjunction morphism $M\to K_R \res M$ is invertible if and only if $M$ is stable and $\Ext^1(N, M)$ vanishes for all $N $ which lie in the kernel of $\res$. 

\item Dually, $K_{L} \res M \to M$ is invertible if and only if $M$ is costable and $\Ext^1(M, N)$ vanishes for all $N$ which lie in the kernel of $\res$. 
\end{enumerate}
\end{lemma} 
Note that, as a consequence of (3), the intermediate extension $K_{LR} $ establishes an equivalence between 
$ \Mod  \cS$ and the full subcategory of $\Mod \cR$ with objects the bistable modules.

In this article we construct a functor from representations of the
singular Nakajima category $\cS$ to $\cp$, establishing a bijection 
between the strata of the affine quiver variety and the objects of $\cp$ (see Theorem~\ref{thm: stratification-dynkin}). 

To this end, we define for every $\cS$--module $M$ the modules $CK(M)$ and $KK(M)$ given by
\begin{align*}
KK(M) &=\ker(K_L(M) \to K_{LR}(M)) \mbox{ and } \\
CK(M) &=\cok(K_{LR}(M) \to K_R(M)).
\end{align*}
Note that both $CK$ and $KK$ are supported only in $\cR_0 -\cS_0$. 
Now we have an obvious isomorphisms 
$\cR/ \langle \cS \rangle \iso \cp$. Therefore,
we may view $CK(M)$ and $KK(M)$ as $\cp$--modules. If $Q$ is of Dynkin type, we have an isomorphism $\cR^{gr}/ \langle\cS^{gr} \rangle \iso \text{Ind} \cd_Q$, the category of indecomposable objects 
in $\cd_Q$. The respective functors in the graded setting,  which we denote here by $CK^{gr}$ and $ KK^{gr}$, can be seen as functors from $\Mod \cS^{gr}$ to $\Mod \cd_Q$.

\subsection{Preprojective algebras} \label{ss: preprojective}
The {\em double quiver $\overline {Q}$} is obtained from $Q$ by adding, for each arrow $\alpha$, 
a new arrow $\overline \alpha$ with inverted orientation. For example, if $Q$ is the quiver $1 \to 2$, the
double quiver is
\[ 
\xymatrix{ 
1 \ar@/^/[r] &
2 \ar[l] .
} 
\]

The preprojective algebra, denoted by $\cp_Q$, is the path algebra $k \overline{Q}/  I$ where $I$ 
is the ideal generated by the mesh relations $r_x$ for all $x \in Q_0$. Recall that 
\[
r_x = \sum_{\beta: y \to x} \overline{\beta} \beta- \sum_{\alpha: x \to y} \alpha \overline{\alpha} 
\]
is the {\em mesh relation} associated with a vertex $x$ of $Q$, where the
sum runs over all arrows of $Q$.

Note that by \cite{Lusztig91} Section 12.15, the preprojective algebra depends up
to isomorphism only on the underlying graph of the quiver. Hence we can
assume, that Q has a bipartite orientation, which means that the mesh
relations are of the form
$$
r_x =  \sum_{\beta: y \to x} \overline{\beta} \beta \text{ or}  \sum_{\beta: x \to y} \beta \overline{\beta}
$$

These are exactly the relations of $\cR /  \langle \cS \rangle$, where we choose $C$ to be all vertices
of $\Z Q$ and $F=\tau$ as in the example \ref{exa: A2}.

It will also be convenient to view $\cp_Q$ as the mesh category associated to $\overline{Q}$, that is the objects of $\cp_Q$ are 
the vertices of $Q$ and the morphisms are the $k$-linear combinations of paths in $\overline Q$ modulo the ideal generated by
the mesh relations $r_x.$ From this point of view, the categories of indecomposable objects in $ \proj \cp_Q $ and $\cp_Q$ are equivalent categories by the Yoneda embedding. 

%
%

\section{Generalized Nakajima quiver varieties}\label{ss: generalized}
Nakajima introduces in \cite{Nakajima01} two types of quiver varieties, 
the affine quiver variety $\cm_0$ and the smooth quiver variety $\cm$.
Both quiver varieties give rise to a graded version: the graded affine quiver variety and the graded quiver variety,
which we denote in this paper by $\cm^{gr}_0$ and $\cm^{gr}$.

We define generalized quiver varieties as geometric invariant theory quotients of representations of the regular Nakajima category $\cR$. To obtain the definition 
of the graded quiver variety it suffices to replace $\cR$ by $\cR^{gr}$. We refer to \cite{KellerScherotzke13a} for more details.

\begin{remark} 
For $F=\tau$ and $Q$ acyclic, the original Nakajima varieties are obtained as $\cm(v,w)$ and $\cm_0(w)$. For $F=\tau^n$ the $n$-cyclic quiver varieties are obtained as $\cm(v,w)$ and $\cm_0(w)$.
\end{remark}

We fix two dimension vector $v: \cR_0 -\cS_0 \to \N$ and $w:   \cS_0 \to \N$ and set 
${\mathcal St}(v,w)$ to be the subset of $\rep( v,w, \cR)$ consisting of 
all $\cR$--modules with dimension vector $(v, w)$ that are stable in the sense of  \ref{ss:Kan extensions and Stability}.
Let $G_v$ be the product of
the groups $\GL(k^{v(x)})$, where $x$ runs through the
non-frozen vertices. By base change in the spaces
$k^{v(x)}$, the group $G_v$ acts on the
set ${\mathcal St}(v,w)$. The {\em generalized quiver
variety $\cm(v,w)$} is the quotient
${\mathcal St}(v,w)/G_v$. As shown in the next statement, the action of $G_v$ is free,  and therefore
$\cm(v,w)$ is well-defined. Because of the assumption on $F$, graded quiver varieties are not a special case of generalized quiver varieties. 

The next statement has been proven for classical, graded and cyclic quiver by Nakajima \cite{Nakajima01} \cite{Nakajima11} for the
case where $Q$ is Dynkin or bipartite and by Qin \cite{Qin12} \cite{Qin13a} 
and Kimura-Qin \cite{KimuraQin12} for the extension to the case of an 
arbitrary acyclic quiver $Q$. We provide a proof for all generalized quiver varieties.
\begin{theorem}
The $G_v$--action is free on ${\mathcal St}(v,w)$ and 
the quasi-projective variety $\cm(v,w)$ is smooth. 
\end{theorem}
\begin{proof}
Suppose that $g M= M$ for some $g \in G_v$ and $M \in {\mathcal St}(v,w)$. 
Then $\im ( \id-g )M $ is a submodule of $M$ which has support only on
objects of $\cR_0-\cS_0$. This is a contradiction to the stability of $M$. 
As the action of $G_v$ is free, it remains to show that ${\mathcal St}(v,w)$ is smooth.  
We consider the map 
\[
 \nu: \rep(v,w, \tilde Q) \to \bigoplus_{x \in \cR_0} \Hom( \C^{v(x)} , \C^{v(\tau x)}) ,
M \mapsto  \bigoplus_{ x \in   \cR_0-\cS_0} \sum_{\alpha: y \to x} M(\bar  \alpha)M( \alpha) 
\] where the sum runs through all arrows $\alpha:  y \to x$ in $\tilde Q$ and $\bar \alpha$ denotes the unique 
arrow $ \tau x \to y$ in $\tilde Q$ corresponding to $\alpha$. 
The set $\nu^{-1}(0) $ yields the affine variety $ \rep(v,w, \cR)$. Furthermore the tangent space at 
a point $M \in \rep(v,w, \cR)$ corresponds to the space of all $N \in \rep(v,w, \tilde Q)$ satisfying that 
\[
d\nu_M(N) =  \bigoplus_{ x \in   \cR_0} \sum_{\alpha:  y \to x}   N(\bar \alpha) M(\alpha) +  M (\bar \alpha) N(\alpha)
\]
vanishes. 
We show that in every mesh to $ x\in\cR_0- \cS_0$, and every $f \in \Hom(\C^{v(x)}, \C^{v(\tau x)})$ there is a  representation $N \in \rep(v,w, \tilde Q)$ with image $f$. Clearly, if $M_{x}: M(x) \to M(\sigma(x))$ is injective, this is possible by choosing $N_{\sigma(x)} :M(\sigma x)  \to M(\tau x) $ such that $ N_{\sigma(x)}  M_x=f$ and all other linear maps of $N$ to be zero. Then $d\nu_M(N)=f$. Assume now that $\ker M_x \not =0 $ and $s \in \ker M_x$. 

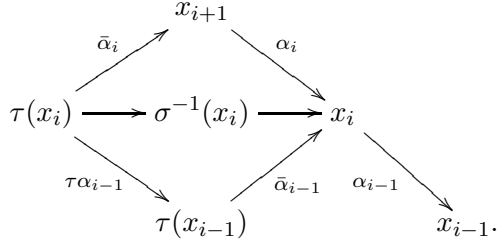
\[ 
\xymatrix{ &  x_{i+1} \ar[rd]^{ \alpha_i}&  &   \\
\tau(x_i)  \ar[r]  \ar[ru]^{\bar \alpha_i}   \ar[rd]_{\tau \alpha_{i-1}} &   \sigma^{-1} (x_i) \ar[r]  & x_i \ar[rd]_{\alpha_{i-1}}  & \\
& \tau(x_{i-1})  \ar[ru]_{\bar \alpha_{i-1}}  &  &x_{i-1} . }
\]

Then, by the stability of $M$, there is a vertex $z \in \cR_0 -\cS_0$ such that the image of $s$ along a path from $z$  to $x$ in $\tilde Q$ does not lie in the kernel of $ M_z: M(z) \to M(\sigma(z))$. Let us assume that $z=x_n, \ldots , x_0=x$ are the vertices along a minimal length path with arrows $\alpha_i: x_{i+1} \to x_{i} $ and that $f$ is a projection of $s$ onto $s'\in M(z)$.  Then, we can choose inductively along the path for every arrow of the path, linear maps $N(\bar  \alpha_i) $ such that we get components 
$N(\bar \alpha_0) M(\alpha_0)=f$ and 
$M(\tau  \alpha_{i-1}) N( \bar \alpha_{i-1}) + N(\bar  \alpha_{i})  M( \alpha_{i})=0$ for all $1\le i \le n-1$. In the last mesh, we choose $N_{\sigma(z)}: M(z) \to M(\sigma(z))$ such that $N_{\sigma(z)}  M_{z}  + M(\tau \alpha_{n-1}) N(\bar \alpha_{n-1}) =0$. All other linear maps defining $N$ are set to zero.  By the definition of $N$, we have $d\nu_M(N)=f$ and $d\nu_M$  is therefore surjective. 
Hence the tangent spaces at stable points are of 
the same dimension. We conclude that ${\mathcal St}(v,w)$ and $\cm(v,w)$ are smooth. 
\end{proof}

\begin{remark}
Note that generalized quiver varieties are in general not symplectic. Furthermore in general they do not satisfy the odd cohomology vanishing 
proved in Theorem 7.3.5 of \cite{Nakajima01} for classical quiver varieties.
\end{remark}

The {\em affine quiver variety} $\cm_0(w)$ is constructed as follows: we consider for a fixed 
$w$ the affine varieties given by the categorical quotients $$\cm_0(v,w):=\rep(v, w, \cR)//G_v.$$ For all $v' \le v$, where the order is component-wise, 
we have an inclusion by extension by zero, that is by adding semi-simple nilpotent representations with dimension vector $v-v'$  
\[
\rep(v', w, \cR)// G_{v'} \hookrightarrow \rep(v, w, \cR)// G_{v}.
\]
The affine quiver variety $\cm_0(w)$ is defined as the colimit of the quotients $\rep(v, w, \cR)// G_v$ over all $v$ along the inclusions.

If $Q$ is a Dynkin quiver,
it follows from Lemma \ref{finite} and Lemma \ref{lemma:closed-orbits}, that for $v,\ v'$ big enough, the
map $\rep(v',w, \cR)//G_{v'} \hookrightarrow \rep(v,w, \cR)//G_{v}$ is indeed an equivalence and
therefore the affine quiver variety is well-defined in the category of schemes.
In the case that $Q$ is not of Dynkin type, we view $\cm_0(w)$ as an Ind-scheme.

\begin{lemma} \label{finite} Let $Q$ be of Dynkin type. Then $St(v,w, \cR)$ vanishes for all
but finitely many dimension vector $v$.
\end{lemma}
\begin{proof} Let $M$ be a stable representation. Then for every $i\in \cR_0-\cS_0$ and $ 0 \not = a \in M(i)$
 there is a frozen vertex $\sigma(j)$ and a path from $\sigma(j)$ to $i$ such that
the induced linear map $M(i) \to M(\sigma(j))$ does not contain $a$ in its kernel.
We can assume without loss of generality that there is a minimal length path $p_a$ from $j$ to $i$ which does not vanish in $\cp$ such that the composition with the arrow $\alpha_j: \sigma(j) \to j$ satisfies the condition. 
This can be seen as follows: if $p_a$ vanishes in $\cp$ then it is equivalent in $\cR$ to  
a sum of paths which all contain a frozen vertex and are of the same length than $p_a$. Hence
 we can find a shorter path which satisfies the condition. 
By Lemma 2.8 there are only finitely many non-zero paths in $\cp$. It follows that
there is an injective map
$$
 M(i) \to  \bigoplus_{p \in e_i \cp} M(\sigma(j))
$$
and hence for  fixed $w$, there are only finitely many dimension vectors $v$ such
that $St(v,w, \cR)$ does not vanishes.
\end{proof}

\begin{theorem} The set $\cm(v,w)$ canonically becomes a smooth quasi-projective variety and the projection map
\[
\pi_v: \cm(v,w) \to \cm_0(v,w)
\]
taking the $G_v$--orbit of a stable $ \cR$--module $M$ to the unique closed $G_v$--orbit in the closure of $G_v M$, is proper. 
If $Q$ is of Dynkin type, then the induced map $\pi: \sqcup_{v} \cm(v,w) \to \cm_0(w)$ is also proper. 
\end{theorem}
\begin{proof} The fact that $\cm(v,w)$ is smooth is proven in Theorem 3.2. Let $\chi : G_v \to  C$ be the determinant character. By GIT theory, see Theorem
2.2.4 and Section 2.2 in \cite{Ginzburg}, we have that
$$\rep(v, w, \cR) //_{\chi}G_{v } \to \rep(v,w, \cR)//G_v$$
is a projective morphism and the set of stable points in $\rep(v, w, \cR)//_{\chi} G_v$ 
given by $\cm(v,w)$ form a Zariski open subset. Hence the map  $$\cm(v,w) \to \cm_0(v,w)$$  is proper. 
If $Q$ is of Dynkin type,  $\cm(v,w)$ vanishes on all but finitely many
dimension vectors $v$ and  $ \cm_0(v,w)$ embeds into $\cm_0(w)$. Hence the result
follows.
\end{proof}

\begin{definition}\label{strata}
We denote by $\cm^{reg}(v,w) \subset \cm(v,w) $ the open set consisting of the union of closed $G_v$--orbits of stable representations. 
Furthermore, we denote $\cm_0^{reg}(v,w) := \pi(\cm^{reg}(v,w))$ the open possibly empty subset of $\cm_0(v,w)$. We refer to $\cm_0^{reg}(v,w) \subset \cm_0(v,w)$ as a stratum of $ \cm_0(v',w) $ via the embedding $\cm_0(v,w) \hookrightarrow \cm_0(v',w) $ for $v' \ge v$. 
Note that it follows from the definition that $\pi$ is one-to-one when restricted to $\cm^{reg}(v,w)$. 
\end{definition}

\subsection{Affine quiver varieties and $\cS$--modules}
Based on \cite{LeclercPlamondon12}, we have proven in \cite{KellerScherotzke13a} 
that the graded affine quiver variety $\cm^{gr}_0(w)$ is equivalent to $\rep(\cS^{gr},w)$ as schemes. In this section, we show that this result remains
true for classical quiver varieties with $Q=A_n$. We give a sufficient condition
for this equivalence to be true in the generalized setting. In general the
equivalence fails and we only have a natural bijection between the sets of 
closed points of $\cm_0(w)$  and $\rep(\cS^{gr},w)$.

First, we proceed by describing the closed points of $\cm_0(w)$ more concretely. By abuse of language, we say that 
a $G_v$-stable subset of $\rep(v, w, \cR)$ {\em contains} a module,
if it contains the orbit corresponding to the module.

\begin{lemma}\label{lemma:closed-orbits}
\begin{itemize}
\item The closed $G_v$--orbits in $\rep(v, w,  \cR)$ are represented by 
$ L \oplus N \in \rep( v,w, \cR)$, where $L$ is a bistable module and $N$ is a semi-simple module such that $\res N$ vanishes. 

\item A stable $\cR$--module $M$ belongs to $\cm^{reg}(v,w)$ if and only if it is bistable.
\end{itemize}
\end{lemma}
\begin{proof}
(1) Given an 
exact sequence
\[
\xymatrix{ 
0 \ar[r] & N \ar[r] & M \ar[r] & L \ar[r] & 0
}
\]
of finite-dimensional $\cR$--modules with $\res(N)=0$ ( resp. $\res(L)=0$ ), the closure of the 
$G_v$-orbit of $M$ contains $N_{ss}\oplus L$  ( resp. $L_{ss}\oplus N$ )
where $N_{ss}$ is the semi-simple module with the same composition series than $N$.

The
module $K_{LR}(\res M)$ is contained in $\im(\eps) \subset K_R(\res M)$, where 
$\eps$ denotes the adjunction morphism $M \to K_R(\res M)$.
Let $i$ denote the inclusion $K_{LR}(\res M) \subset \im(\eps)$.
Now by the first step, the closure of the orbit of $M$ contains
$\im(\eps)\oplus (\ker(\eps))_{ss}$ and the closure of the orbit
of $\im(\eps)$ contains $K_{LR}(\res M)\oplus (\cok(i))_{ss}$.
Hence if the $G_v$--orbit of $M$ is closed, it contains the object
\[
K_{LR}(\res M)\oplus (\cok(i))_{ss} \oplus (\ker(\eps))_{ss}.
\]
We conclude that $M \cong K_{LR}(\res M)\oplus (\cok(i))_{ss} \oplus (\ker(\eps))_{ss} $ and 
as $K_{LR} \res M$ is bistable, this proves the first implication. 

Conversely, consider a point $L \oplus N$ as in the statement of the Lemma. As $N$ is semi-simple, every element in the closure 
of the orbit of $L \oplus N$ contains $N$ as direct summand. By GIT theory, we know that 
the closure of the orbit of $L \oplus N$ contains a closed orbit $G_v X$.  By the first part $X \cong K_{LR}( \res X) \oplus Z \oplus N$, where $\res Z$ vanishes. 
As the restriction functor is $G_v$--invariant and algebraic, it takes constant value 
on the orbit closure. Therefore $\res L \cong \res X$ and as $L$ is bistable, we have $L \cong K_{LR} (\res L) \cong K_{LR} (\res X)$. Furthermore,
for dimension reasons, $Z$ vanishes. Hence the orbit of $L \oplus N$ and the orbit of $X$ coincide. This finishes the proof. 

(2) Suppose that $M$ is stable and belongs to $\cm^{reg}(v,w)$, then by the first part $M \cong  K_{LR}(\res M)\oplus N$, where 
$\res N $ vanishes. As $M$ is stable, it forces $N=0$. Hence $M$ is bistable by Lemma~\ref{lemma:stable} and conversely, by the first part, all bistable modules give
rise to closed $G_v$--orbits. 
\end{proof}
If $Q$ is of Dynkin type, then $\cS$ is a finitely-generated algebra and the
restriction induces algebraic maps
$$ \res :\cm_0(w) \to \rep(w, \cS)  \text{ and } \res  \pi:\cm(v,w) \to  \rep(w, \cS). $$

Remarkably, in the graded setting, the varieties $\cm^{gr}_0(w)$ and $ \rep(w, \cS^{gr})$ are always isomorphic for any choice of $Q$. This is  proven in \cite{LeclercPlamondon12} and \cite{KellerScherotzke13a}. 
We can show that, if $Q$ is of Dynkin type $A$ and $F= \tau$, then $\rep(w, \cS) $ and
$\cm_0(w)$ are equivalent affine schemes.

\begin{theorem}\label{thm:bijection}
If Q is of Dynkin type, then the natural map
$$ \cm_0(w) \to  \rep(w,  \cS)$$
induces a bijection on the closed points.

\end{theorem}
\begin{proof} By \cite{LiP} and \cite{Lu}, the ring of invariants is generated by coordinate
maps to paths from a frozen vertex to a frozen vertex and trace maps to
cyclic paths in non-frozen vertices. Let us choose $v$ such that $\cm_0(v,w)\simeq \cm_0(w).$ 
Then there is a map from the coordinate ring of $\rep(w, \cS)$ to
the coordinate ring of $\cm_0(v,w)$, inducing the restriction map $\cm_0(v,w) \to \rep(w, \cS)$.
Applying the colimit over $v$ yields a morphism of $\cm_0(w)$ to $\rep(w,\cS)$. This
map is a bijection on the closed points by Lemma \ref{lemma:closed-orbits} as all semi-simple
$\cp$-modules are nilpotent if $Q $ is of Dynkin type by Lemma \ref{lemma: tilde P}. 
\end{proof}
The next example shows that the map of affine schemes $\cm_0(w) \to  \rep(w, \cS)$
is in general not an equivalence.

\begin{example}
Let us consider the example of $Q = A_2$. Then $\cR^{gr}$ is given by 
\[ \xymatrix
{\ldots \ar[r]  & P_2 \ar[r] \ar[dr] &  \sigma(\Sigma S_2) \ar[r] & \Sigma S_1  \ar[r] \ar[dr] & \ldots  & \\
S_1 \ar[ur] \ar[r] & \sigma(S_2) \ar[r] &S_2 \ar[ur] \ar[r] & \sigma(\Sigma P_2) \ar[r] &  \ldots &
}
\] 
and we choose $F$ to be the automorphism which maps $S_1$ to $P_2$. Then $\cR$ is
given by the path algebra of

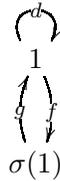
\[ \xymatrix{ 1 \ar@/^/[d]|f   \ar@(ul,ur)| d\\
\sigma(1) \ar@/^/[u] |g   }
\] 
with relation $d^2 = gf$. Furthermore $\cS$ is given by the path algebra 
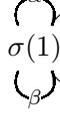
\[ \xymatrix{ \sigma(1)   \ar@(ul,ur)| \alpha  \ar@(dl,dr)| \beta
}
\] with relation $\beta^3=\alpha^2$. We denote by $Tr(c)$ the function, which maps a
representation to the trace of the endomorphism ring associated to any cycle
$c$ in the quiver. The invariant ring of $\C[\rep(v, 1,\cR)]^{G_w}$ has generators $Tr(c)$,
where $c$ is any cycle that passes through $\sigma(1)$ and $Tr(d)$. We have the
relation $Tr(d)^2 = Tr(fg)$. The coordinate ring $\C [\rep(v, 1, \cR)]^{G_w}$ is given by
$$\C[x, y,z]/ (y -zx, z^2 -x) \simeq \C[y, z]/(y-z^3)$$
for $v \not=0$ and the coordinate ring of $\rep(1,\cS)$ is
$$\C[x, y]/ (y^2- x^3) \simeq \C[z^2,y]/(y^2 -z^6).$$
Hence we see that $\rep(1, \cS)$ is a cusp, in particular singular, while $\cm_0(1)$ is
its normalisation.

We consider $St(v,1, \cR)$. As there are exactly two non-zero paths in $\cp$ given by the path of length zero and $d$, we obtain, that 
$ St(v,1, \cR)$ is not empty if and only if $v=1$ or $v=2$. 
\end{example}

For $F=\tau$ and $Q=A_n$, we have an equivalence of schemes though.

\begin{theorem}\label{thm:An}
If $Q = A_n$ and $F=\tau$, then the natural map
$$\cm_0(w) \to \rep(w,  \cS)$$
induces an equivalence of schemes.
\end{theorem}
\begin{proof} We show that for any cycle $c$ with vertices passing only through non-frozen
vertices, we have $Tr(c) \in \C[\rep(w, \cS)]$. We proceed by induction on
the number of different vertices contained in the cycle. We give $A_n$ a linear
orientation and number the vertices by $1, \ldots, n$ linearly. Let us denote by
$c_{i,i+1}$ the 2-cycle with vertices $i,i + 1$ in $\overline{Q}$. 
Suppose, we have a cycle 
$c^l_{i,i+1}$ and $ l \in \N$. 
Then by the relations in $\cR$, we have  $c_{i,i+1} = -c_{i,i-1} -c_{i, \sigma(i)}$.
Hence by recursion, it is enough to retreat to the case, where $i = 1$ and $l = 1$.
But then $Tr(c_{2,1}) = Tr(c_{1,2}) = Tr(c_{1, \sigma(1)})$, which finishes the case. Let us
assume that $c$ contains vertices from $i$ to $i + k$. Then 
$Tr(c) = Tr(c^l_{i+1,i} c')$
where $c'$ has vertices $i + 1, \ldots, i + k$.

 Applying recursion, we can rewrite
 $c_{i+1,i} = -c_{i+1,i+2}- c_{i+1, \sigma{(i+1)}}$. Hence $Tr(c)$ is the sum of $Tr(c^l
_{i+1, i+2}c')$ and trace functions of cycles passing through a frozen vertex. Now by induction,
the first term can be written as the trace of a sum of cycles passing through
at least one frozen vertex.
\end{proof}
The previous Theorem fails in general for Dynkin quivers. A counterexample
is the case $Q = D_4$, $F=\tau$ and dimension vectors with value $2$ in
every vertex. We give a sufficient condition for the equivalence of scheme to
hold.

\begin{theorem} If $Q$ is a Dynkin quiver and $F$ an automorphism such that
$\cd_Q(P_i, F^nP_i) = 0$ for all indecomposable projective $Q$-modules $P_i$ and all
$n \in \Z-\{0\}$, then the natural map
$$\cm_0(w) \to \rep(w, \cS)$$
induces an equivalence of schemes.
\end{theorem}
\begin{proof} The condition
$$\cd_Q(P_i, F^nP_i) = 0$$
is equivalent to the fact that all oriented cycles in $\cp$ vanish. Hence every
non-zero cycle in $\cR$ is equal to a cycle that passes through a frozen vertex.
Then all trace functions are already contained in the coordinate algebra of
$\rep(w,\cR)$, and hence the map
$$\cm_0(w) \to \rep(w, \cS)$$
is an equivalence of schemes.
\end{proof}

For later use we record the next corollary which is an easy consequence of the previous results in this section. 

\begin{corollary} \label{corollary: points in affine qv}
Every closed point $ L \in\cm_0(v,w) $ corresponds uniquely to a pair $(L_1, L_2)$ where $ L_1$ is bistable  and $L_2$ is a representative of the isomorphism class of a semi-simple $ \cp$--module, such that $L_1 \oplus L_2$ has dimensionvector $(v,w)$. With this identification the map $\pi : \cm(v,w) \to \cm_0(v,w)$ is given by $G_v N \mapsto (K_{LR} (\res N), N')$ where $N'$ is the semi-simple module with the same composition series than 
$$ \coker( K_{LR} (\res N) \to N ).$$ 
\end{corollary}
\begin{proof}As the closed points in $\cm_0(v,w)$ are given by  the closed $G_v$-orbits of $\rep(v,w, \cR)$, the first statement follows by  Lemma \ref{lemma:closed-orbits}. 
The map $\pi$ maps the $G_v$-orbit of a stable module $N$ to the unique closed orbit in the closure of $G_vN$. The unique closed orbit in the closure of $G_v N$ is by the proof of part (1) of Lemma \ref{lemma:closed-orbits}  the orbit of  $K_{LR} \res  N \oplus \coker(i)_{ss} \oplus \ker(\epsilon)_{ss}$, where $\epsilon: N \to K_R(\res N)$. As $N$ is stable, $\epsilon$ is injective. Therefore $\ker(\epsilon)$ vanishes, $ \im(\epsilon)\simeq N$ and $i$ is given by $K_{LR} (\res N) \to N$. Furthermore, $K_{LR} \res N$ is bistable by Lemma \ref{lemma:stable}. 
\end{proof}

If $Q$ is of Dynkin type, all semi-simpe $\cp$-modules are nilpotent, and by the definition of $\cm_0(w)$ the closed points $(L_1, L_2)$ and $(L_1, 0) $ get identified. As a consequence, 
we obtain that every closed point of $\cm_0(w)$ lies in a strata $\cm^{reg}_0(v,w)$ as defined in \ref{strata}. 

\subsection{From graded to ungraded Nakajima categories}
In Nakajima's original construction, the graded and cyclic quiver varieties are defined as fixed point sets of the quiver varieties with respect to a $\C^*$-action (see \cite{Nakajima01}). Hence they are naturally closed subvarieties of the original quiver varieties. We used a different approach in this paper: we started by introducing the graded quiver varieties as moduli spaces of graded Nakajima categories and introduced quiver varieties as moduli spaces 
of orbit categories of the graded Nakajima categories. In this section, we will relate the two different approaches.

We denote $p: \cR^{gr}_C \to \cR$ the pushforward functor, 
$$p_*: \Mod \cR_C^{gr} \to \Mod \cR$$ 
 and 
$$p^*: \Mod \cR \to \Mod \cR^{gr}_C$$ the pullback functor. 
More concretely, they are given by 
$$p_*M (x) = \bigoplus_{n\in \Z} M(F^n x) \text{ for all }M \in \Mod \cR_C^{gr} $$ and 
$$p^*(N)(x) =N(p(x)) \text{ for all }N \in \Mod \cR.$$ 
The same functors exists at the level of $\cS^{gr}_C$ and $\cS$--modules. As $p$ commutes with the restriction functor, we will also denote the pushforward and pullback functors by 
$p_*: \Mod \cS^{gr} \to \Mod \cS$ and $p^*: \Mod \cS \to \Mod \cS^{gr}$ respectively. 
It is easy to see that the pushforward functor 
induces an embedding of graded quiver varieties into classical quiver varieties. 
To obtain the embedding of cyclic quiver varieties into classical quiver varieties, one considers the pushforward functor to the canonical functor $ \cR^{gr}/ \tau^n \to \cR\cong \cR^{gr}/ \tau$.


We call an $\cR^{gr}$--module (resp. an $\cS^{gr}$--module) $M$ {\em left bounded }if there is an $n \in \Z$ 
such that $M(m, i)=0 $ for all $m \le n$ and all vertices  $i$ of $Q^f_0$. Right bounded modules are defined analogously and  a 
bounded module is left and right bounded. 

\begin{lemma}\label{adjointness for p}
The functors $(p_*, p^*)$ are adjoint functors. 
If $A$ is a bounded $\cR^{gr}_C$--module (resp. $\cS^{gr}_C$--module) and $B$ a $\cR$--module (resp. $\cS$--module),  we also have a natural isomorphism 

$$\Hom(B, p_*A) \cong \Hom(p^*B, A).$$

\end{lemma}
Recall from \cite{KellerScherotzke13a} that all indecomposable projective (resp. injective) modules are given by $x^{\wedge}$ (resp. $x^{\vee}$) for all objects $x$. Furthermore, they are pointwise finite-dimensional and right bounded (resp. left bounded). 

Both the pullback and pushforward functors are exact. Hence $p_*$ maps finitely generated projective modules of the graded Nakajima category to finitely generated projective modules of the generalized Nakajima category. 


In the next Lemma, we investigate the relationship between the Kan extensions in the graded and non-graded case. 
\begin{lemma}\label{lemma:kan-extensions-graded-ungraded}
Let $M$ be an $ \cS^{gr}_C$--module, then $p_* K^{gr}_L M \cong K_L p_* M$ and there is a canonical inclusion 
\[
i: p_* K^{gr}_R M \hookrightarrow K_R p_* M
\] whose restriction to $\cS$ is the identity. If $Q$ is of Dynkin type and $M$ is finite-dimensional, then $i$ is an isomorphism. 
\end{lemma}
\begin{proof}
To prove the first isomorphism, we use the adjointness relations
 \begin{eqnarray*}
 \Hom(p_* K^{gr}_L M, L) &\cong& \Hom( K^{gr}_L M , p^* L) \cong \Hom( M, \res p^* L) \\
 &=&  \Hom(p_* M, \res L)\cong \Hom(K_L p_* M, L)
\end{eqnarray*}
 for any $ L \in \Mod \cR$. Hence by the uniqueness of the left adjoint we find $p_* K^{gr}_L M \cong K_L p_* M$. 
Finally, we have 
\[
 \Hom(S, p_* K^{gr}_R M) \hookrightarrow  \Hom( p^* S , K^{gr}_R M) =0,
 \] and hence we have that $p_* K^{gr}_R M$ is stable and as $\res p_* K^{gr}_R M \cong M$ there is a canonical injection $p_* K^{gr}_R M \to K_R p_* M$ which restricts to the identity on $\cS$. If $Q$ is of Dynkin type and $M$ is finite-dimensional, then $K^{gr}_R M$ is bounded and we obtain
\begin{eqnarray*}
 \Hom(L, p_* K^{gr}_R M) &\cong & \Hom( p^* L, K^{gr}_R M ) \cong \Hom(p^* \res  L, M) \\
 &= & \Hom( \res L, p_* M) \cong \Hom(L, K_R p_* M ) 
\end{eqnarray*} for any $ L \in \Mod \cR$. Hence by the uniqueness of the right adjoint we find $p_* K^{gr}_R M \cong K_R p_* M$. 
\end{proof}
We obtain the following easy consequence.

\begin{lemma}\label{lemma: KK}
For $M \in \mod \cS^{gr}_C$, we have 
\[K_{LR} p_* M \cong p_* K^{gr}_{LR} M \text{ and } 
p_*KK^{gr}(M)\cong KK(p_* M).
\]
\end{lemma}
\begin{proof}
It follows from the previous proof that $K_{LR} p_* M$ is the image of the canonical map $p_* K^{gr}_L M \to p_* K^{gr}_R M$ and by the exactness of 
$p_*$, the image is $p_* K^{gr}_{LR} M$. 
By Lemma~\ref{lemma:kan-extensions-graded-ungraded}, the functor 
$p_*$ commutes with the  left Kan extension. Hence the second statement follows from the exactness of $p_*$. 
\end{proof}

\subsection{Homological properties of Nakajima categories}
First we determine the projective resolutions of nilpotent simple $\cR$--modules. For an object $x \in \cR_0$, we denote by $S_x$ the simple module given by the
functor that sends the object $x$ to $k$ and every non-isomorphic object to
zero.

\begin{lemma} \label{lemma: resolutions-of-simples-in -R}
Let $x$ be a non-frozen vertex.
\begin{itemize}
\item[a)] The nilpotent simple $\cR$--modules have projective resolutions given by

$$  0 \to   \tau(x)^{\wedge} \to \bigoplus_{y \to x} y^{\wedge} \to x^{\wedge} \to S_x \to 0$$ 
where the sum runs over all arrows $y \to x$ in $\tilde Q$.
If $x \in C$, then the projective resolution of $S_{\sigma x} $ is given by $$ 0 \to \tau(x)^{\wedge} \to \sigma(x)^{\wedge} \to S_{\sigma(x)} \to 0 .$$  

\item[b)] The nilpotent simple $\cR$--modules have injective resolutions given by
$$  0 \to S_x  \to  x^{\vee} \to \bigoplus_{x \to y} y^{\vee} \to \tau^{-1}(x)^{\vee} \to 0$$ where the sum runs over all arrows $x \to y$ in $\tilde Q$. 
If $x \in C$, then the injective resolution of $S_{\sigma x} $ is given by 
$$ 0 \to S_{\sigma(x)} \to {\sigma(x)}^{\vee} \to x^{\vee} \to 0 .$$
\end{itemize}
\end{lemma}
\begin{proof}
We apply $p_*$ to the projective resolution of simple $\cR^{gr}$--modules given in \cite{KellerScherotzke13a}. As $p_*$  is an exact functor, we obtain the above exact sequence. 
Now applying a dual argument to $  \cR^{op}$ yields the injective resolutions.
\end{proof}
  
Next, we determine projective resolutions of simple nilpotent $\cS$--modules. 
Let $\cc$ be either $ \cS$ or $ \cR$. 
Let us denote by 
$$P(x):= \bigoplus_{y \in C/F} \cp(y,x)\otimes \sigma(y)^{\wedge} $$ the projective $\cc$--module and by 
$$I(x):= \prod_{y \in C/F} D \cp(x,y) \otimes \sigma(y)^{\vee} $$ the injective $\cc$--module which we associate to an object $x \in C$.
The graded analogues, where we replace $\cp$ with $\cd_Q$
 shall be denoted $P^{gr}(x)$ and $I^{gr}(x)$.
\begin{lemma}\label{lemma: homology of S}
Let $x \in C$.
\begin{itemize}
\item[a)] Let $Q$ be of Dynkin type, then the nilpotent simple $ \cS$--modules  have infinite projective resolutions
\[
 \cdots \to P(\Sigma^2 \tau x) \to P(\Sigma \tau x) \to P(\tau x) \to \sigma(x)^{\wedge} \to S_{\sigma(x)} \to 0 .
 \] 

\item[b)] If $Q$ is not of Dynkin type, we obtain a projective resolution for the simple nilpotent $\cS$--modules 
\[
 0 \to  P(\tau x) \to \sigma(x)^{\wedge} \to S_{\sigma(x)} \to 0.
 \]  
 
Dualizing these sequence yields injective resolutions. 
\end{itemize}
\end{lemma} 
\begin{proof}
We apply $p_*$ to the projective resolution of the simple $\cS^{gr}$--module $S_{\sigma(x)}$ given in \cite{KellerScherotzke13a} and we use that 
\[
p_*(P^{gr}( x) )= \bigoplus_{y \in \cd_Q} \cd_Q (y, x) \otimes p_*\sigma(y)^{\wedge}
\]
and 
\[
 \bigoplus_{l \in \Z} \cd_Q(F^l(y), x) \cong \cp(y, x).
\] Dually, we obtain the injective resolution. 
\end{proof}
We summarize some consequences of the previous Lemmata. 
\begin{corollary}\label{corollary: global dimension}
\begin{itemize} 
\item The category of finite-dimensional nilpotent $\cR$--modules has global dimension $2$. 
\item If $Q$ is of Dynkin type, the 
category of finite-dimensional nilpotent $\cS$--modules has infinite global dimension and it is hereditary if $Q$ is not of Dynkin type.
\end{itemize}
\end{corollary}
Note also that as an algebra, $\cS$ is  finitely generated if $Q$ is of Dynkin type.
The results on the homological properties of $\cS$ allow us to describe $\cS$ 
as a quotient $kQ_{\cS}/I$, where $Q_{\cS}$ is a finite quiver and 
$I$ an admissible ideal of the path algebra $k Q_{\cS}$ cf.~\cite[Ch.~8]{GabrielRoiter92} \cite[II.3]{AssemSimsonSkowronski06}. 
We determine the quiver $Q_{\cS}$ and the number of minimal relations between two vertices. 
\begin{proposition}\label{proposition: quiver of S}
Let $Q$ be of Dynkin type. Then the quiver $Q_{\cS}$ has vertices $ x\in C$ and the number of arrows $x \to y $ is equal to $\dim   \cp(y, \Sigma x)$. The number of minimal relations of paths from $x $ to $y$ is given by the dimension of $\cp(y, \Sigma^2 x)$. 
\end{proposition}
\begin{proof}
As all objects in $\cS$ are pairwise non-isomorphic, the vertices of $Q_{\cS}$ are in bijection with the objects of $\cS$, which are in bijection via $\sigma$ with the objects in $C/F$. Let us hence denote the vertices of $Q_{\cS}$ by the objects in $C$.
The number of arrows from $x$ to $y$ is the dimension of $ \Ext^1_{\cS}(S_{\sigma(x)}, S_{\sigma(y)}) $ and the number of minimal relations between paths from $x$ to $y$ is the dimensions of $\Ext^2_{\cS}( S_{\sigma(x)}, S_{\sigma(y)})$ which by Lemma~\ref{lemma: homology of S} is given by  the dimensions of $\cp(y, \Sigma x)$ and $\cp(y, \Sigma^2 x)$ respectively.
\end{proof}
We determine $\cS$  in terms of a quiver with relations for two different choices of functors $F$. 
\begin{example}
Let $Q$ be the $A_2$ quiver $1 \to 2$.  We choose $F=\tau$ and in this case $Q_{\cS}$ is given by 
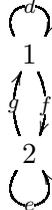
\[ \xymatrix{ 1 \ar@/^/[d]|f   \ar@(ul,ur)| d\\
2 \ar@/^/[u] |g    \ar@(dl,dr)| e}
\] 
The minimal relations defining $\cS$ are $d^3-fg=0$, $e^3-gf=0$, $df-fe=0$ and $eg-gd=0$. To obtain the minimal relations one computes that $\cp_Q(x,y) $ is one dimensional 
for any choice of vertices $x$ and $y$. It follows that there is exactly one minimal relation between paths joining any two vertices $x$ and $y$. As $\cS$ is an orbit category of $\cS^{gr}$, the exact minimal relations can also be obtained using the Example of \cite{KellerScherotzke13a} after Lemma 2.7. 
\end{example}

\begin{example}
Let $F=\Sigma \tau^{-1} $, then $\cp$ is the cluster category and $\cS$ is given as the path algebra of the quiver 
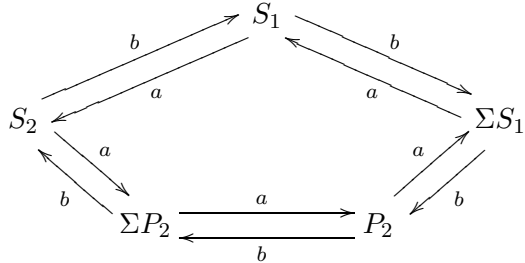
\[
\xymatrix{ 
& &S_1  \ar@<1ex>[lld]^a  \ar@<1ex>[rrd]^b && \\
S_2 \ar@<1ex>[dr]^a  \ar@<1ex>[rru]^b&& && \Sigma S_1  \ar@<1ex>[ld]^b  \ar@<1ex>[llu]^a\\
& \Sigma P_2 \ar@<1ex>[ul]^b \ar@<1ex>[rr]^a & & P_2 \ar@<1ex>[ll]^b    \ar@<1ex>[ru]^a &
}
\] subject to the relations $ab=ba$, $a^3=b^2$ and $b^3=a^2$. 
\end{example} 

We can classify simple $\cR$--modules in terms of simple $\cp$ and simple $\cS$--modules. 
\begin{lemma}
The intermediate extension $K_{LR}$ establishes a bijection between simple $\cS$--modules and all simple $\cR$--modules which do not vanish under restriction to $\cS$. 
\end{lemma}
\begin{proof}
Let $S$ be a simple $\cS$--module. Let $S'$ be a non-trivial simple submodule $S' \le K_{LR} S$. As $K_{LR} S$
is stable, so is $S'$ and we find that $$ 0 \not = \res S' \le \res K_{LR} S \cong S$$ and hence $\res S' \cong S$. Furthermore $S'$ is costable as 
it is simple and has non-zero support in frozen vertices. As $K_{LR} S$ is the unique bistable module up to isomorphism whose restriction is isomorphic to $S$, we have that $S' \cong K_{LR} S$. 
As every simple $\cR$--module whose restriction to $\cS$ is non-zero is bistable, we have that $K_{LR} S$ is the unique simple lift of $S$. 
Now let $L$ be a simple $\cR$--module. Clearly $\res L$ vanishes if and only if $L$ is supported only in non-frozen vertices. Suppose that $\res L$ does not vanish. Then $L$ is stable and $K_{LR} (\res L)$ is a submodule of $L$. Hence $K_{LR} (\res L) \cong  L$ is simple. Suppose that $L'$ is a simple submodule of $\res L$. Then $ K_{LR} (\res L')$ is also a simple submodule of $L$ and by the identity 
\[
L' \cong \res K_{LR} (\res  L')  \cong \res L
\] we have that $L' \cong \res L$. 
\end{proof}
\begin{remark} Note that all simple $\cR$--modules are finite-dimensional, as $\cR$ is finitely generated. It follows from the previous Lemma, that the same
holds for simple $\cS$--modules.
\end{remark} 
\section{The Stratification functor of affine quiver varieties}

For $Q$ of Dynkin type, we construct a functor
$$\Psi: \mod \cS \to \proj \cp$$
which parametrizes the strata in the following way: two points $M$ and $M_0$
lie in the same stratum of $\cm_0(w)$ only if they are mapped to isomorphic
images under $\Psi \circ \res$. Furthermore, we show that 	
$\Psi$ describes the degeneration order between strata.
Recall that the strata of $\cm_0(w)$ are given by the images of $\cm^{reg}(v,w)$ under $\pi$ (see Definition \ref{strata} ). 

\begin{lemma}
Let $M \in \mod \cS$. Then the functor $KK$ and $CK$ satisfy
\[
\Ext^1( K_{LR} M, S) = \Hom(KK (M) , S) \text{ and } \Ext^1(S,  K_{LR} M) = \Hom(S, CK (M) )
\]
for all $\cR$--modules $S$ which are supported in $\cR_0 -\cS_0$. 
Furthermore $\Ext^1(KK (M), S )$ and $\Ext^1(S, CK (M))$ vanish for all finite-dimensional nilpotent modules $S$ which are supported in $\cR_0 -\cS_0$. 
\end{lemma}
\begin{proof}
We give the proof for the functor $KK$, the proof for $CK$ being dual. 
By applying $\Hom( -, S)$ to the short exact sequence  $$ 0 \to KK( M) \to K_L M   \to K_{LR} (M) \to 0$$ we obtain the exact sequence
\begin{eqnarray*}
\Hom( K_L M, S)  & \to &  \Hom(KK (M), S) \to \Ext^1( K_{LR} M, S)  \\ 
  &\to& \Ext^1(K_L M , S) \to \Ext^1(KK (M), S) \to \Ext^2(K_{LR} M, S). 
\end{eqnarray*}
By Lemma~\ref{lemma:stable} we have that $\Hom( K_L M, S)$ and $\Ext^1(K_L M , S)$ vanish, proving the first identity.  

We deduce from Lemma~\ref{lemma: resolutions-of-simples-in -R} that 
\[
 \Ext^2(K_{LR} M, S) \cong \Hom(S, K_{LR} M).
 \] The second term vanishes as $K_{LR} M$ is stable, which proves the second statement.
\end{proof}
\begin{lemma} \label{lemma:projectivity-of-KK}
Let $Q$ be of Dynkin type. 
For any $M \in \mod  \cS$, the $\cp$--module $KK(M)$ is finitely generated and projective. Dually, $CK(M)$ is a finitely cogenerated injective $\cp$--module. 
\end{lemma}
\begin{proof} 
As $M$ is finitely-generated the same is true for $K_L(M)$ and its submodule $KK(M)$. As $Q$ is of Dynkin type, the algebra $\cp$ is finite-dimensional. Hence
$KK(M)$ is finite-dimensional and finitely presented as $\cp$--module. Furthermore all simple $\cp$--modules are nilpotent and hence isomorphic to $S_x$ for some $x \in \cR_0 -\cS_0$. As $\Ext^1(KK(M), S_x)$ vanishes by the 
previous Lemma for all simple $\cp$--modules $S_x$, the module $KK(M)$ is a finitely generated projective $\cp$--module. The proof that $CK(M)$ is finitely cogenerated injective is dual. 
\end{proof}

If $Q$ is of Dynkin type, then $\proj \cp$ carries a triangulated structure, which lifts from the triangulated structure of $\cd_Q$. Hence the shift functor $\Sigma$ seen as an automorphism of $\proj \cp$ makes $\proj \cp$ a triangulated category. This is proven for instance in \cite{Keller05}  Section 7.3, where $\proj \cp$ is seen as an orbit category of $\cd_Q$ by the Auslander-Reiten translation $\tau$.

\begin{proposition}\label{prop: CK and KK}Let $Q$ be of Dynkin type, then $CK$ is equivalent to $ \Sigma KK$.
  \end{proposition}
\begin{proof}
Note that the Serre functor on $\cd_Q$ is given by $\tau \Sigma$ and induces the Serre functor of $\proj \cp$ and an automorphism on $\cp$. We will also denote the
automorphism on $\cp$  by $\tau \Sigma$. 
The multiplicity of $x^{\wedge} $ appearing as a direct summand of $KK(M)$  is given by the dimension of $\Hom( KK(M), S_x)\cong \Ext^1(K_{LR} M, S_x)$. Dually the multiplicity of $x^{\vee}=(\tau \Sigma x)^{\wedge} $ as a direct summand of $CK(M)$ is given by the dimension of $$\Hom(S_x,  CK (M))\simeq \Ext^1(S_x, K_{LR} M).$$ Now as the dimensions of $\Ext^1(S_x, K_{LR} M)$ and $\Ext^1( K_{LR} M, S_{\tau x})$ are equal, both numbers coincide, which finishes the proof. 
\end{proof}

The next Lemma shows that, if $Q$ is of Dynkin type, the functors $KK$ and $CK$ are essentially surjective onto $\proj \cp$. 

\begin{lemma}\label{lemma: KK-of-simple}
We have $KK(S_{\sigma(x)} ) \cong x^{\wedge}_{\cp}$  and  
$CK(S_{\sigma(x)} ) \cong x^{\vee}_{\cp}$ for all $x \in C$.
 \end{lemma}
 \begin{proof} 
To prove the first identity we use the result of Lemma~\ref{lemma: KK}. 
We have shown in \cite{KellerScherotzke13a}, that $KK^{gr}(S_{\sigma(x)}) $ is a projective module of $\cp$ represented by the image of the Happel functor $H(x)$.
It follows that $$KK(S_{\sigma(x)})=p_*(\Hom_{\cd}( -, H(x)))= x^{\wedge}_{\cp}.$$  

To prove the second identity, we observe that by Lemma~\ref{lemma:stable} the module $K_R(\sigma(x)^{\vee})$ is isomorphic to $ \sigma(x)^{\vee}$ seen as $\cR$--module. This follows from the fact that $\sigma(x)^{\vee}$ is stable and $\Ext^1(-,  \sigma(x)^{\vee})$ vanishes. 
Hence $K_R I(x)$ is isomorphic to $I(x)$ seen as $\cR$--module. 
We consider now the commutative diagram with exact rows and exact columns
\[\xymatrix{ 
 \ker g \ar@{^{(}->}[r] \ar@{^{(}->}[d] & S_{\sigma(x)} \ar@{^{(}->}[d] \ar[r]& 0 \ar@{^{(}->}[d] \\
 K_R S_{\sigma(x)} \ar@{^{(}->}[r] \ar@{->>}[d]^{g} &\sigma(x)^\vee \ar[r] \ar@{->>}[d] & I(x) \ar@{=}[d] \\
 x^\vee_{\cp} \ar@{^{(}->}[r] & x^{\vee} \ar[r] &  I(x)}
 \]
The exactness of the middle row follows from applying $K_R$ to the beginning of the injective resolution of $S_{\sigma(x)} $ in Lemma~\ref{lemma: resolutions-of-simples-in -R}. To obtain the exactness of the last row, we use Theorem 3.7 of \cite{KellerScherotzke13a}. 
Hence it follows that $\ker g$ equals $S_{\sigma(x)} \cong K_{LR} (S_{\sigma(x)})$ and therefore $CK(S_{\sigma(x)} ) \cong x^\vee_{\cp}$. 
\end{proof}

Recall that a $\delta$--functor is a functor from an exact category to a triangulated category, which sends short exact sequences 
to distinguished triangles up to equivalence, cf. e.~g. \cite{Keller91}.

\begin{proposition}
Let $Q$ be of Dynkin type. The functors $KK$ and $CK$ are $\delta$-functors 
$$ \mod \cS \to \proj  \cp .$$ 
\end{proposition}
\begin{proof}

Let us show first that $CK$ is a $\delta$--functor. Let $0 \to M \to N \to L \to 0$ be an exact sequence in $\mod \cS$. 
Then we obtain the following commutative diagram
 \[ \xymatrix{ 
  CK(M) \ar[r] & CK(N) \ar[r] & CK( L)  \\
  K_R(M) \ar@{^{(}->}[r] \ar@{->>}[u]& K_R( N) \ar@{->>}[r] \ar@{->>}[u]& K_{R}( N)/ K_R(M) \ar[u]\\
 K_{LR} M \ar@{^{(}->}[r]  \ar@{^{(}->}[u] & K_{LR} N \ar@{^{(}->}[u]  \ar@{->>}[r] & K_{LR} N/ K_{LR} M  \ar[u]_f  \\
}\]
As $K_{R}( N)/ K_R(M) $ is stable, it embeds into $K_R(L)$. 

Furthermore, as $K_{LR} N/ K_{LR} M$ is costable, the image of $f$ in $K_R(L) $ is given by $K_{LR} L$. 
Hence, we have that $ \ker(f) \cong \ker( K_{LR} N/ K_{LR} M \to K_{LR} L) $ and as $K_{LR} N/ K_{LR} M$ is costable and restricts under $\res$ to $L$, the canonical map $K_L(L) \to K_{LR} L$ factors through $K_{LR} N/ K_{LR} M \to K_{LR} L$. Hence, we obtain the following commutative diagram 
\[ \xymatrix{  KK(L) \ar@{^{(}->}[r] & K_L( L) \ar@{->>}[r] & K_{LR} L\\
 S \ar@{^{(}->}[r] \ar@{^{(}->}[u]& K_L( L) \ar@{->>}[r] \ar@{=}[u]& K_{LR} N/ K_{LR} M   \ar@{->>}[u] \\
}\]
Applying the snake Lemma, we conclude that $KK(L) \cong \Sigma^{-1} CK(L) $ maps onto $\ker(f) \cong KK(L) /S$. 
We consider next the diagram 
\[ \xymatrix{  S \ar@{^{(}->}[r] & K_L( L) \ar@{->>}[r] & K_{LR} N/ K_{LR} M  \\
KK(N) \ar@{^{(}->}[r] \ar[u]& K_L( N) \ar@{->>}[r] \ar@{->>}[u]& K_{LR} N \ar@{->>}[u]  \\
}\]
As $K_{LR} M$ is costable, there is no map from $\ker (K_{LR} N \to K_{LR} N/ K_{LR} M) $ to $S$. 
So by the snake Lemma the map $KK(N) \cong \Sigma^{-1} CK(L)$ maps surjectively onto $S$. 

We conclude that $$\Sigma^{-1} CK(N) \to \Sigma^{-1} CK(L) \to CK(M) \to CK(N) \to CK(L) $$ is an exact sequence in 
$\proj \cp$ and as an exact sequence, 
$$\Sigma^{-1} CK(L) \to CK(M) \to CK(N) \to CK(L) $$ is therefore isomorphic to a distinguished triangle in $\proj  \cp$.  
Hence $CK$ and $KK$ give rise to a $\delta$-functor. 
\end{proof}

Note that, as $CK$ is a $\delta$-functor, we are able to recover the image of $CK$ on $\cS$--modules from the image  of  $CK$ on simple $\cS$--modules.

For a vector $v: \cR_0 -\cS_0 \to \Z$, we define $C_q v : \cR_0 -\cS_0   \to \Z$ by
\[
(C_q v)(x) = v(x) - \big(\sum_{y\to x} v(y)\big) + v(\tau(x)),\  x\in \cR_0 -\cS_0,
\]
where the sum ranges over all arrows $y\to x$ of $\tilde Q$ and $y \in  \cR_0 -\cS_0 $.

We will also denote by $w \sigma: \cR_0 -\cS_0 \to \N$ the dimension vector which sends $ x$ to $ w(\sigma(x)) $ if $x \in C$, and vanishes otherwise. 
Using $C_q$, we can compute explicitly the functors $KK$ and $CK$. 
\begin{proposition}
 Let $Q$ be of Dynkin type and let $K_{LR} M$ have dimension vector $(v,w)$, then the multiplicity of $z_{\cp}^{\wedge}$ as direct summand 
of $KK(M)$ is given by $(w \sigma  - C_q v )(z) $.   
\end{proposition}
\begin{proof}
Let $M \in \mod \cS$. 
We have already shown in Lemma~\ref{lemma:projectivity-of-KK} that $KK(M)$ is a finitely generated projective $\cp$--module and 
that $KK(S_{\sigma(x)})\cong x^{\wedge}_{\cp}$ in Lemma~\ref{lemma: KK-of-simple}. 
Hence it remains to show that the isomorphism class of $KK(M)$ depends only on $\gdim K_{LR} M$. 
The multiplicity of $x^{\wedge}_{\cp}$ as a direct summand of $KK(M)$ is given by $\dim \Hom(KK(M), S_x)= \dim \Ext^1( K_{LR} M , S_x)$. 
Now the dimension $\Ext^1( K_{LR} M , S_x)$ is given by the dimension of the cohomology of the right and left exact complex 
$$ 0 \to K_{LR} M (\tau x) \to \bigoplus_{x \to y}K_{LR} M (y) \to K_{LR} M (x) \to 0$$  
where the sum ranges over all arrows $x \to y$ of $\overline{ Q}$. Hence the dimension
equals  $w \sigma -C_q v$ and depends uniquely on the dimension vector of $K_{LR} M$
\end{proof}

If $Q$ is of Dynkin type and $ \cp$ isomorphic to the preprojective algebra $\cp_Q$,
the map $C_q$ is injective. Indeed in this case, the map $C_q$ corresponds to the Cartan matrix associated with the diagram of $Q$. 
We obtain a stronger statement. 
\begin{theorem}\label{thm: stratification-dynkin} 
Let $Q$ be of Dynkin type and  let $C_q$ be injective. The functors $KK$ and $CK$ are $\delta$--functors 
$$ \mod \cS \to \proj  \cp$$ satisfying 
that any two modules $M_1$ and $M_2$ belonging to $\cm_0(w)$ lie in the same stratum if 
and only if their image under $CK \circ \res $ respectively $KK \circ \res$ are isomorphic. 
\end{theorem}
\begin{proof}  
Let $M \in \cm_0(w)$. We recall that the dimension vector of $K_{LR}M$ determines the stratum of $M$: 

We have that $M\in \cm_0(v,w) $ if and only if  $\gdim K_{LR} ( \res M)=(v,w)$ by Lemma~\ref{lemma:stable}. 
We have shown in the previous proposition that the dimension vector of $K_{LR}( \res M)$ determines the image of $KK(\res M)$ and $CK(\res M)$ up to 
isomorphism. 

So suppose that $M_1, \ M_2 \in \cm_0(w)$ and lie in the strata $\cm_0^{reg}(v_1,w)$ and $\cm^{reg}_0(v_2, w)$ respectively. 
If $KK(\res M_1) \cong KK(\res M_2)$ then $w\sigma -C_qv_1=w\sigma -C_qv_2$, and as $C_q$ is injective, we find $v_1=v_2$.  
Since by Proposition \ref{prop: CK and KK},  we have that $CK =  \Sigma KK$, the analogous statement also holds for $CK$. 
\end{proof}

In the non-Dynkin case, we have a similar statement, relating the strata to the isomorphism classes of the objects in $\inj^{nil} (\cp)$,  which denotes the category of finitely cogenerated injective modules with nilpotent socle. Let us denote by $I_x$ the injective hull of $S_x$ seen as a $\cp$--module.

\begin{theorem}\label{thm: stratification}
There is a map $$\Psi:\mod \cS \to \inj^{nil} ( \cp)$$ such that two objects appearing in the same stratum are isomorphic under 
$\Psi \circ \res $. We have  $\Psi(S_{\sigma(x) })= I_x$ for all $x\in \cR_0-\cS_0$. 
\end{theorem}
\begin{proof}
Let $I_M$ be the injective hull of $CK(M)$. As $CK(M)$ is finitely cogenerated
for finitely cogenerated $M$, we know that $I_x$ appears only finitely 
many times as a factor of $I_M$. We define $\Psi(M)$ to be the maximal direct factor of $I_M$ with nilpotent socle. As $K_{LR} S_{\sigma(x)}= S_{\sigma(x)}$, we have that $\dim \Ext^1(S_z, K_{LR} S_{\sigma(x)}) =\delta_{z,x}$ and hence $\Psi(S_{\sigma(x)})=I_x$. 
Proving that this defines a map parametrizing the strata of $\mod \cS$ is analogous to the Dynkin case. 
\end{proof}
Note that if $Q$ is not of Dynkin quiver, the maps $CK$ and $\Psi$ do not coincide: we have that $CK(S_{\sigma(x)} ) = x^\vee_{\cp} $ which is not isomorphic to 
$\Psi(S_{\sigma(x)}) =I_x$. 

\subsection{Degeneration order on Strata}
Let us assume in this section that $Q$ is of Dynkin type and that the map $C_q$ associated with $\cp$ is invertible. 
The correspondence between the finitely generated projective $\cp$--modules and the strata of the affine quiver variety allows us to identify the degeneration order of strata with a degeneration order on objects of the triangulated category $\proj \cp$. 

We recall first the degeneration order on strata of  the classical quiver variety $
\cm_0(w)$:
we define $\cm^{reg}_0(v,w) \le \cm_0^{reg}(v', w)$ if and only if $\cm^{reg}_0(v,w) \subset \overline{\cm_0^{reg}(v', w)}$.
This is shown to be the case if and only if $v(x) \le v'(x) $ for all $ x \in  \cR_0- \cS_0$  by \cite[4.1.3.14]{Qin12}.  

Analogously, we define a partial order on the strata of $\cm_0(w)$. For two points $M, M' \in \cm_0(w)$, we set $M \le M'$
if and only if $v(x) \le v'(x)$ for all $x\in \cR_0 -\cS_0$, where $(v,w)$ and $(v',w)$ are the dimension vectors of $K_{LR} (\res M)$ and $K_{LR}(\res M')$ respectively. 
This clearly defines a partial order on the strata $\cm^{reg}_0(v,w)$ of $\cm_0(w)$, which we also call the degeneration order. 

In \cite{JensenSuZimmermann05a}, Jensen, Su and Zimmermann define a partial order on objects of a triangulated category $\ct$ 
satisfying the following conditions:
\begin{enumerate}
\item $\ct$ is $Hom$-finite and idempotent morphisms are split;
\item For all $X, Y \in \ct$ there is an $n \in \Z-\{0\}$ such that $\Hom(X, \Sigma^n Y) $ vanishes. 
\end{enumerate}
The first condition assures transitivity and the second condition allows to conclude that the preorder is anti-symmetric. 
Given two objects $X$ and $Y$ of $\ct$, we write in convention with our notation, 
$X \le Y$ if and only if there is an object $Z$ in $\ct$ and a distinguished triangle
 $$X \to Y \oplus Z \to Y \to \Sigma X. $$
 
We show that under the stratification functor $KK$, the degeneration order on strata corresponds to the order on objects of 
$\proj \cp$. Note that in the case that $\ct $ is $\cp$ condition $(1)$ is satisfied by Lemma \ref{lemma: tilde P}, but the condition $(2)$ is not satisfied. Hence the order of \cite{JensenSuZimmermann05a} is only a preorder. 
But it will follow from the next Theorem that the preorder defines in fact a partial order on the objects of $\proj \cp$.
\begin{theorem}
If $Q$ is of Dynkin type, the degeneration order among the strata of $\cm_0(w)$
corresponds to the degeneration order of the triangulated category $\proj \cp$.  
\end{theorem}
\begin{proof}
Let $M$ and $M'$ be two elements in $\cm_0(w)$ belonging to the strata $ \cm^{reg}_0(v,w)$ and $ \cm_0^{reg}(v', w)$ respectively. Let us assume that $KK(\res M) \le KK(\res M')$ in the degeneration order of \cite{JensenSuZimmermann05a}. We will show that $v(x) \le v'(x)$ for all $x\in \cR_0-\cS_0$. 

We suppose there is a triangle 
\begin{equation} \label{triangle}
KK(\res M) \to KK(\res M') \oplus Z \to Z \to \Sigma KK(\res M) 
\end{equation}
 for some $Z \in \proj \cp$. 

By Lemma~\ref{lemma: KK-of-simple} we can find semi-simple 
$\cS$--modules $Y$ and $L$ such that $KK(\res M) \cong KK(Y)$ and $KK(L)=Z$. 

Hence $\gdim K_{LR} Y= (0, w-C_qv \sigma^{-1})$ and $\gdim K_{LR} L = (0,w_L)$ for the dimension vector  $w_L=\gdim L$ of $\cS$. 
Note that $KK$ induces an isomorphism 
\[
\Ext^1_{\cS}(S_{\sigma(x)}, S_{\sigma(y)}) \cong {\cp}(x, \Sigma y).
\]
Hence we can lift the triangle \ref{triangle} to an exact sequence
$$ 0\to Y \to N \to L \to 0$$ 
for some $ N \in \mod \cS$.

Applying the functor $K_{LR}$ to the short exact sequence and using the left and right exactness of $K_{LR}$ yields the inequality
$$ \gdim K_{LR} Y  + \gdim K_{LR} L \le \gdim K_{LR} N.$$
Let $\gdim K_{LR} N=(v_N, w_N)$, then we find
$w_N= w-C_qv \sigma^{-1} + w_L$ and $v_N \ge 0$. 

As $KK(N) \cong KK(L) \oplus KK(\res M')$, we have that 
\[w_N  - C_q v_N \sigma^{-1}  = (w_L + w )- C_qv'  \sigma^{-1}.
\] Hence it follows that $ C_q(v'-v_N) =C_qv$ implying  $v'- v_N =v$ by our assumption on $C_q$. 

For the converse claim, we refer to \cite[3.18]{KellerScherotzke13a}.
The proof is analogous if we replace the split Grothendieck group of $\cd_Q$ by the split Grothendieck group of $\proj  \cp$. 
\end{proof}
Clearly, the degeneration order is anti-symmetric. Hence the same holds for the preorder on objects of $\proj \cp$. 

\def\cprime{$'$} \def\cprime{$'$}
\providecommand{\bysame}{\leavevmode\hbox to3em{\hrulefill}\thinspace}
\providecommand{\MR}{\relax\ifhmode\unskip\space\fi MR }
\providecommand{\MRhref}[2]{%
  \href{http://www.ams.org/mathscinet-getitem?mr=#1}{#2}
}
\providecommand{\href}[2]{#2}

\end{document}